\documentclass{article}
\usepackage[utf8]{inputenc}
\usepackage{amsfonts}
\usepackage{amsmath}
\usepackage{commath}
\usepackage{amsthm}
\usepackage{ amssymb }
\usepackage{enumerate} 
\usepackage{verbatim}
\usepackage{hyperref}
\usepackage[pagewise]{lineno}
\usepackage{graphicx}
\usepackage{float}
\usepackage{diagbox}
\usepackage{mathtools}
\numberwithin{table}{section}

\usepackage{color}

\usepackage[utf8]{inputenc}
\usepackage{amsmath}
\usepackage{amssymb}
\usepackage{amsthm}
\usepackage{verbatim}
\usepackage{enumerate}
\usepackage{amsfonts}
\usepackage{amsmath}
\usepackage{hyperref}
\usepackage{comment}
\usepackage{amsthm}
\usepackage{hyperref}
\usepackage{cleveref}
\usepackage[normalem]{ulem}
\usepackage{tikz}
\usetikzlibrary{calc,positioning,decorations.pathmorphing,decorations.pathreplacing}
\usepackage{ifthen}

\usepackage{marginnote}

\newtheorem{theorem}{Theorem}
\newtheorem{conj}[theorem]{Conjecture}
\newtheorem{lemma}[theorem]{Lemma}
\newtheorem{prop}[theorem]{Proposition}
\newtheorem{rem}[theorem]{Remark}

\newtheorem*{remark*}{Remark}
\newtheorem{defin}[theorem]{Definition}

\newcommand{\Z}{\mathbb{Z}}

\usepackage{authblk}  
\usepackage{fancyhdr} 

\title{A linear upper bound on the zero-sum Ramsey number of forests in $\mathbb{Z}_p$}

\author{
    Lucas Colucci\thanks{Instituto de Matemática e Estatística, Universidade de São Paulo, Brazil. \texttt{lcolucci@ime.usp.br}}, 
    Marco D'Emidio\thanks{Colégio Etapa, Brazil \texttt{mandemidio@gmail.com}}
}

\begin{document}

\maketitle

\begin{abstract}
 Let $m$ be a positive integer and let $G$ be a graph. The zero-sum Ramsey number $R(G,\Z_m)$ is the least integer $N$ (if it exists) such that for every edge-coloring {$\chi \, : \, E(K_N) \, \rightarrow \, \Z_m$} one can find a copy of $G$ in $K_N$ such that $\sum_{e \, \in \, E(G)}{\chi(e)} \, = \, 0$. 
 
In this paper, we show that, for every prime $p$,
$$R(F,\mathbb{Z}_p)\leq n+9p-12$$
for every forest $F$ in $n\geq 3p^2-12p+11$ vertices with $p\mid e(F)$ without isolated vertices.
\end{abstract}

\section{Introduction and main result} \label{sec : intro}


Let $k$ be a positive integer and let $G$ be a graph with $k\mid e(G)$. The \emph{zero-sum Ramsey number} $R(G, \mathbb{Z}_k)$ is the smallest integer $N$ such that, for every edge-coloring $\chi: E(K_N) \to \mathbb{Z}_k$, there exists a copy of $G$ in $K_N$ satisfying $\sum_{e \in E(G)} \chi(e) = 0$. We say that such a copy of $G$ has zero-sum. It is simple to see that $R(G, \mathbb{Z}_k)$ exists if, and only if, $k \mid e(G)$.

\medskip

Recall that the $k$-colored Ramsey number $R_k(G)$ is the minimum $N$ such that every $k$-coloring of the edges of $K_N$ has a monochromatic copy of $G$. In particular, $R(G, \Z_k) \leq R_{k}(G)$. For a comprehensive review of zero-sum Ramsey problems, we recommend the survey of Caro~\cite{caro1994complete}. 

\medskip

The exact value of $R(G,\mathbb{Z}_k)$ is known for on a few instances of $G$ and $k$. On the other hand, some general upper bounds are known for some families of graphs. For complete graphs, the following is known (\cite{alon1993three},\cite{caro1992zero}, \cite{caro1994linear}):

\begin{theorem}\label{thm:uppercomplete}
    Let $k$ be an integer and suppose $k \mid \binom{n}{2}$.

\begin{enumerate}
    \item[(i)] \textit{If $k$ is an odd prime-power, and $n+k$ is at least the Ramsey number $R(K_{2k-1}, k)$, then $R(K_n, Z_k) \leq n + 2k - 2$. If, in addition, $k$ is a prime that divides $n$ then $R(K_n, Z_k) = n + 2k - 2$.}
    \item[(ii)] \textit{If $n \geq R(K_{3k-1}, k)$ then $R(K_n, Z_k) \leq n + k(k+1)(k+2) \log_2 k$.}
\end{enumerate}
\end{theorem}

A curious feature of the zero-sum Ramsey number, in contrast to its usual (monochromatic) counterpart, is that, for a fixed $k$, $R(G,\mathbb{Z}_k)$ is not a monotonic function of $G$. That is, it is not true, in general, that if $k\mid e(G), e(H)$ and $H$ is a subgraph of $G$, then $R(H,\mathbb{Z}_k) \leq R(G,\mathbb{Z}_k)$. For instance, $R(C_4,\mathbb{Z}_2)=4<5=R(2K_2,\mathbb{Z}_2)$, where $C_4$ is the cycle on $4$ vertices, and $2K_2$ is the matching on two edges. In particular, Theorem \ref{thm:uppercomplete} does not give bounds for $R(G,\mathbb{Z}_k)$ for non-complete graphs.

Some other notable results on zero-sum Ramsey numbers are the complete determination of $R(G,\mathbb{Z}_2)$ and the determination of $R(F,\mathbb{Z}_3)$ for forests, as the next two results show:

\begin{theorem}\label{thm:exactz2}(\cite{caro1994complete})
For any graph $G$ with $n$ vertices and an even number of edges,
\begin{equation*}
    R(G, \mathbb{Z}_2) =
    \begin{cases}
        n + 2 & \text{if $G = K_n$ and $n = 0,1 \mod{4}$,} \\
        n + 1 & \text{if $G = K_p \cup K_q$ and $\binom{p}{2} + \binom{q}{2} = 0 \mod{4}$, or $G$ is odd,} \\
        n & \text{otherwise, }
    \end{cases}
\end{equation*}
\end{theorem}

    \begin{theorem}\label{thm:upeerz3}(\cite{alvarado2025problem})
    Let $F$ be a forest on $n$ vertices with $3\mid e(F)$ and without isolated vertices. Then,
    \begin{equation*}
    R(F,\mathbb{Z}_3)=\begin{cases}
			n+2, & \text{if $F$ is $1 \pmod 3$ regular or a star;}\\
            n+1, & \text{if $3\nmid d(v)$ for every $v \in V(F)$ or $F$ has exactly one}\\ 
            & \text{vertex of degree $0 \pmod 3$ and all others are $1 \pmod 3$,}\\ 
            & \text{and $F$ is not $1 \pmod 3$ regular or a star;}\\
            n, & \text{otherwise.}
		 \end{cases}
   \end{equation*}
\end{theorem}

Following the line of research suggested by these results, in this paper we prove the following general upper bound on the zero-sum Ramsey number of forests in $\mathbb{Z}_p$:

\begin{theorem}\label{thm:upperprimeforests}
    Let $p$ be a prime number. Then,
    $$R(F,\mathbb{Z}_p)\leq n+9p-12$$
    for every forest $F$ in $n\geq 3p^2-12p+11$ vertices with $p\mid e(F)$ without isolated vertices.
\end{theorem}

\section{Definitions and notation} \label{sec : prelim}

In this section, we give some definitions  that will be used later in the proof.

For a leaf in a forest, we call its only neighbor its \emph{parent}, and the leaf is a \emph{child} of this vertex.

\begin{defin}[Bushy forest]\label{def:switching}
    A forest $F$ is called \textbf{bushy} if it contains at least $2(p-1)$ distinct leaves. Otherwise, it is called \textbf{non-bushy}. Notice that, for $p=2$, every nonempty forest is bushy.
\end{defin}

\medskip

Throughout the paper, $p$ will represent a fixed prime number and $K$ will denote the complete graph on which we are trying to embed our forest. Also, a \emph{coloring} of a graph $G$ always means a coloring of the edges of $G$ whose colors are elements of $\Z_p$, i.e., a coloring is a function ${\chi:E(G) \rightarrow \Z_p}$. From this point on, we will always assume the edges of $K$ are colored by some coloring $\chi$.
\medskip
 \begin{defin}[$b$-colorful vertex]\label{def:altc4}
    If $b$ is a positive integer, a vertex $v \in K$ will be called \textbf{$b$-colorful} if there exists a color such that $v$ is connected to at least $b$ but at most $|K|-b-1$ edges of that color.
\end{defin}

 \begin{defin}[Vibrant coloring]
     We call a coloring of a clique $\chi$ \textbf{vibrant} if it has $p-1$ vertices that are $(3p-5)$-colorful. Otherwise it is called \textbf{non-vibrant}.
 \end{defin}

\begin{defin}[Switcher $C_4$]\label{def:cccolorings}
    A colored $C_4$ (cycle of size 4) in $K$ will be called a \textbf{switcher} if its edges are, in some consecutive order, $\{a,b,c,d\}$ and 
    $$\chi(a)+\chi(b) \neq \chi (c)+\chi(d).$$
\end{defin}
\begin{defin}[Switchable coloring]
    A coloring $\chi$ of $K$ is called \textbf{switchable} if it contains $p-1$ pairwise vertex-wise disjoints $C_4$ that are switchers. Otherwise, call it \textbf{non-switchable}.
\end{defin}
\medskip
As usual, we will use the sum-set notation:
$$A_1+A_2+\cdots+A_n\coloneq \{a_1+a_2+\cdots+a_n\mid (a_1,a_2,\cdots,a_n)\in A_1\times A_2\times\cdots\times A_n\}.$$

\medskip
We will also make extensive use of the following generalization of the classical theorem of Cauchy-Davenport in additive number theory, which can easily be derived from the original result through induction:

\begin{theorem}[Generalized Cauchy-Davenport] \label{thm:CauDav}
    If $A_1,A_2,\cdots,A_n$ are nonempty subsets of $\mathbb{Z}_p$, then
    $$|A_1+A_2+\cdots+A_n| \geq \min(p,|A_1|+|A_2|+\cdots+|A_n|-n+1).$$
\end{theorem}

\section{Proof of the main result} \label{sec: UB}

In this section, we provide the linear upper bound for $R(F,\mathbb{Z}_p)$. We will split the proof between two major cases concerning the number of leafs in the graph, each divided into two subcases according to the structure of the coloring $\chi$.

\subsection{Bushy forests}\label{subsec:1mod3reg}

In this subsection, we will deal with the case when $F$ is bushy, meaning we can select some parents $v_1,v_2,\cdots,v_m$ and some of their leaves, each one with $a_1,a_2,\cdots,a_m$ selected leafs, respectively, in such way that we select a total of $p-1$ leaves (meaning $a_i\geq 1$ for all $i\in\{1,2,\cdots,m\}$ and $a_1+a_2
+\cdots+a_m=p-1$) and no $v_i$ is among the selected leaves.
\begin{rem}
    We asked for a bushy forest to have at least $2(p-1)$ leafs instead of simply $p-1$ in order to deal with the case when two leaves are both parent and children of each other (meaning a connected component that is just an edge). Thus, we would be able to select a set of $p-1$ leaves that doesn't contain their parents (which will prove useful in the future).
\end{rem}
We now consider the following subcase:

\subsubsection{The coloring $\chi$ is vibrant}\label{subsec: vib}

In this case, we choose $m$ vertices $u_1,u_2,\cdots,u_m$ among the $p-1$ vertices of $K$ that are {$(3p-5)$-colorful}. For each $u_i$, let $x_i$ be the color that incides on it with at least $3p-5$ and at most $|K|-3p+4$ edges. Also, denote by $S_i = \mathbb{Z}_{p}-\{x_i\}$, notice that $u_i$ neighbors at least $3p-5$ and at most $|K|-3p+4$ edges with color in $S_i$. Now we prove the following lemma: 
\begin{lemma}
    There exists pairwise disjoints sets $X_1,X_2,\cdots,X_m,Y_1,\cdots,Y_m$ of vertices of $K$ that satisfy all of the following:
    \begin{itemize}
        \item $|X_i| = |Y_i| = a_i$;
        \item The edges connecting $u_i$ to the vertices of $X_i$ are all colored $x_i$ and the edges connecting $u_i$ to the vertices of $Y_i$ are \textbf{not} colored $x_i$;
        \item $u_i\notin X_j,Y_k$ for all $1\leq i,j,k,\leq m$.
    \end{itemize}
\end{lemma}

\begin{proof}
    We will choose each pair $(X_i,Y_i)$ inductively:
    \begin{enumerate}
        \item For $i=1$ just take $X_1$ and $Y_1$ as $a_1-$sized subsets of the neighbors of $u_1$ different from $u_2,\cdots,u_m $that are connected to it with color $x_1$ and with some color in $S_i$, respectively. Notice that this is possible since both sets have at least $ (3p-5)-(m-1) \geq (3p-5)-[(p-1)-1] = 2p-3\geq a_1$ vertices to choose from;
        \item Now, suppose $(X_1,Y_1),\cdots,(X_{k-1},Y_{k-1})$ have all been chosen for some $1\leq k\leq m$. We will show how to select $(X_{k},Y_{k})$.

        Since the sets are pairwise disjoint, the union of all $X_i$ with all $Y_j$ for $1\leq i,j\leq k-1$ has exactly:
$$\sum_{i=1}^{k-1}|X_i|+\sum_{i=1}^{k-1}|Y_i| = 2(a_1+a_2+\cdots+a_{k-1})$$

Thus, the number of vertices of $K$ that connect with $u_k$ with color $x_k$ that aren't in this union (this is necessary to meet the disjoint condition) nor in the set $\{u_1,\cdots,u_{k-1},u_{k+1},\cdots,u_m\}$ (this is necessary to meet the third condition) is at least:
$$3p-5-2(a_1+\cdots +a_{k-1})-(m-1)\geq 3p-5-2(p-1-a_{k})-(p-2) \geq a_{k}$$

Therefore, we can choose $X_k$ to be an $a_k$-sized subset of this set. Analogously, we can select $Y_k$.
    \end{enumerate}

Through this induction, we select all the $2k$ sets satisfying the conditions.
\end{proof}
Now, define $X \coloneq \bigcup_{i=1}^m X_i$ and $Y \coloneq \bigcup_{i=1}^m Y_i$ be the union of all $X_i$ and $Y_j$, respectively. Notice that $|X| = |Y|= a_1+\cdots+a_m = p-1$.

\medskip

Notice that $|F|+|X|  =|F|+|Y| =  n+(p-1)\leq |K|$, thus, we can embed $F$ on $K$ in such way that each of the $v_i$ and $u_i$ coincide and each of the $a_i$ of $v_i$'s children we have selected lie on the set $X_i\cup Y_i$ for all $i=1,\cdots,m$.  Considering all such embeddings, our goal is to prove the following:
\medskip
\begin{prop}\label{CauDavProf1}
    One of these embeddings of $F$ is zero-sum.
\end{prop}
\begin{proof}
    For each $v_i = u_i$, let $A_i$ be the set of all the possible values of the sum of the edges joining $u_i$ and its $a_i$ selected children when they vary in $X_i\cup Y_i$.

If $X_i = \{p_{i,1},p_{i,2},\cdots,p_{i,a_i}\}$ and $Y_i = \{q_{i,1},q_{i,2},\cdots,q_{i,a_i}\}$ such that the edges $p_{i,j}u_i$ have color $x_i$ for each $j$ and $q_{i,j}u_i$ have color in $S_i$ for each $j$. 

Now, consider the $a_i$ pairs $\mathcal{P}_{i,j} = (p_{i,j},q_{i,j})$ for $j =1,2,\cdots,a_i$. On one hand, by selecting one of $u_i$'s children as one of the elements in $\mathcal{P}_{i,j}$ for each $j$, we obtain $2^{a_j}$ choices of how to embed the children whose possible sums of the $a_i$ edges connecting them to $u_i$ lie on the set $\mathcal{P}_{i,1}+\mathcal{P}_{i,2}+\cdots+\mathcal{P}_{i,a_i}$, meaning:
$$\mathcal{P}_{i,1}+\mathcal{P}_{i,2}+\cdots+\mathcal{P}_{i,a_i}\subseteq A_i$$

On the other hand, since $|\mathcal{P}_{i,j}| = 2$, we can apply the generalized version of the Cauchy-Davenport (Theorem $\ref{thm:CauDav}$):
\begin{align*}
    |A_i|
    \geq |\mathcal{P}_{i,1}+\mathcal{P}_{i,2}+\cdots+\mathcal{P}_{i,a_i}| 
    &\geq \min(p,|\mathcal{P}_{i,1}|+|\mathcal{P}_{i,2}|+\cdots+|\mathcal{P}_{i,a_i}|-(a_i-1))\\
    &\geq \min (p,2+2+\cdots2-(a_i-1))\\
    &\geq \min(p,a_i+1) = a_i+1.
\end{align*}

\medskip

Now, the set of all the possible values of the sum of the edges in the embeddings of $F$ we described earlier is precisely the set $$A \coloneq A_1+A_2+\cdots+A_{m}.$$

Therefore, we again apply theorem $\ref{thm:CauDav}$ to bound its size by
\begin{align*}
    |A| 
    &\geq \min(p,|A_1|+\cdots+|A_m|-(m-1))\\
    &\geq\min(p,(a_1+1)+\cdots+(a_m+1)-m+1)\\
    &= \min(p,a_1+\cdots+a_m)\\
    &= \min(p,p) = p.
\end{align*}
    
Proving $A$ spans over all the residues modulo $p$, which implies that $0\in A$ and so there must be a embedding of $T$ that has sum $0$.
\end{proof}
\begin{rem}\label{refbushyvib}
This proof actually shows the better bound (and optimal, as shown in Section \ref{sec : conclu}) $$R(F,\mathbb{Z}_p)\leq n+p-1,$$

whenever $F$ is bushy and $\chi$ is vibrant, with no constraints on the number of vertices of $F$ whatsoever.
\end{rem}

\subsubsection{The case when $\chi$ is non-vibrant}\label{subsec: nvibrant}

In this case $F$ has at most $p-2$ vertices that are $(3p-5)-colorful$. Therefore, there is an induced subgraph $K'$ of $K$ with $|K|-(p-2)$ vertices such that for all color $c\in\mathbb{Z}_p$ and $v\in K$ they either share at most $3p-5$ or at least $|K'|-3p+4$ edges. However, by pigeonhole principle, each $v$ shares at least $\frac{|K'|-1}{p}  = \frac{|K|-(p-1)}{p} =\frac{n+9p-12-(p-1)}{p}\geq \frac{3p^2-12p+11+9p-12-(p-1)}{p}\geq 3p-4$ with some color $c_v$. Thus, $c_v$ must have at least $|K'|-(3p-4)$ edges connecting to $v$. Call the color $c_v$ $v's$ \textbf{dominant} color and let $\alpha \coloneq 3p-4$.
\medskip

Now, for $r = 0,1,\cdots,p-1$, let $G_r \subseteq K'$ be the set of the vertices with dominant color $r$. We will prove the following lemma:

\begin{lemma}\label{sqsumGi}
    $|G_0|^2+|G_1|^2+\cdots+|G_{p-1}|^2 \geq |K'|^2-2\alpha|K'|. $
\end{lemma}
\begin{proof}
    The idea is to count the total number of edges of $K'$ that join some $G_i$ and $G_j$ for $i\neq j$. Let $E$ be this set of edges.

    One one hand, each $v \in G_r$ for some $r =0,1,\cdots,p-1$ neighbors at least $|K'|-|G_r|-\alpha$ edges colored $r$ in $E$. Altogether, each component $G_r$ will emit at least $|G_r|(|K'|-|G_r|-\alpha)$ edges in $E$. Thus
\begin{align*}
    |E| &\geq  \sum_{r=0}^{p-1}|G_r|(|K'|-|G_r|-\alpha)\\
    & = (|K'|-\alpha)\sum_{r=0}^{p-1}|G_r|-\sum_{r=0}^{p-1}|G_r|^2 \\
    &= (|K'|-\alpha)|K'|-\sum_{r=0}^{p-1}|G_r|^2 
\end{align*}

    On the other hand, among all edges in $K'$, $E$ only doesn't contain the ones that connect vertices among the components $G_r$, each with $\tbinom{|G_r|}{2}$ internal edges. Therefore:
\begin{align*}
        |E| =& \tbinom{|K'|}{2}-\sum_{i=0}^{p-1}\tbinom{|G_r|}{2} \\
        =& \frac{|K'|^2}{2}-\frac{|K'|}{2}-\frac{1}{2}\sum_{r=0}^{p-1}|G_r|^2+\frac{1}{2}\sum_{r=0}^{p-1}|G_r| \\
        =& \frac{|K'|^2}{2}-\frac{|K'|}{2}-\frac{1}{2}\sum_{r=0}^{p-1}|G_r|^2+\frac{1}{2}|K'| \\
        =&\frac{|K'|^2}{2}-\frac{1}{2}\sum_{r=0}^{p-1}|G_r|^2
        \end{align*}

    Combining both equations:

   \begin{align*}
        &\frac{|K'|^2}{2}-\frac{1}{2}\sum_{r=0}^{p-1}|G_r|^2 = |E| \geq (|K'|-\alpha)|K'|-\sum_{r=0}^{p-1}|G_r|^2 \\
        \implies&\sum_{r=0}^{p-1}|G_r|^2 \geq |K'|^2-2\alpha |K'|
   \end{align*}

\end{proof}

Now let, $M \coloneq \max(|G_r|)$ be the size of the largest $G_r$. We now derive the following claim:
\begin{prop}
    $M\geq |K'|-2\alpha = |K|-6p+8$
\end{prop}
\begin{proof}
    Since $M\geq |G_r|$ for all $r\in0,1,\cdots,p-1$ we have that:

    $$\sum_{r=0}^{p-1}|G_r|^2\leq \sum _{r=0}^{p-1}M|G_r| = M\sum_{r=0}^{p-1}|G_r| = M|K'|$$

    However, because of $\ref{sqsumGi}$ we conclude that:

    $$M|K'| \geq \sum_{r=0}^{p-1}|G_r|^2\geq |K'|^2-2\alpha|K'| \implies M \geq |K'|-2\alpha= |K|-6p+8$$
\end{proof}

Now, we focus on $l$ such that $|G_l|=M$. We finish this case with the following proposition:
\begin{prop}
    $G_l$ contains a monochromatic copy of $F$.
\end{prop}
\begin{proof}
We will rely on the fact that $|G_l| = M \geq |K|-6p+8 \geq n+\alpha$.

We will obtain the desired monchromatic copy through the following greedy algorithm:

\begin{enumerate}
    \item Start by selecting one of the trees $T$ of $F$;
    \item Pick one of $T$'s leaves and place it on one of $G_l$'s vertices;
    \item If we have already embedded a subtree $T_0$ of $T$, pick a vertex of $T$ that was not selected already $v$ that is connected to some vertex $u$ in $T_0$ (note $u$ must be unique for each $v$, else $T$ would have a cycle);
    \item By pigeonhole principle, among the $|G_l|-(|F|-1)\geq \alpha+1$ non-selected vertices of $G_l$ at least one must connect to $u$ (actually the vertex of $G_l$ $u$ was embedded on) with color $l$. Thus, we can embed $v$;
    \item Go back to step 3 until all of $T$ was embedded;
    \item Go back to step 1 until all of $F$ was embedded.
\end{enumerate}
\end{proof}

Hence, there is a monochromatic copy of $F$ in $G_l$. Since $p\mid E(F)$ this copy must also be zero-sum.
\begin{rem}\label{refbushynonvib}
    In this case, we actually used all of the size conditions we imposed about $|K|$ and $n$.
\end{rem}
\subsection{Non-bushy forests}\label{subsec:upper22}
In this section we deal with the case when $F$ is non-bushy. Since for $p=2$, $F$ is always bushy, we will suppose from now on that $p\geq 3$. The main idea will be to focus on the vertices that have degree $2$. For that, we start off with the proposition:

\begin{prop}
    $F$ has at least $n-4p$ vertices of degree $2$.
\end{prop}
\begin{proof}
    Let $n_1,n_2$ and $n_3$ be the number of vertices of $F$ that have degree 1,2 and at least 3, respectively.

    Notice $n_1+n_2+n_3 = n$ and, since $F$ is non-bushy, $n_1\le 2p-3$. Using the handshake lemma,
\begin{align*}
    2e(F) = \sum_{v\in F}\deg v&\geq n_1+2n_2+3n_3 
    \\&= n+n_2+2n_3
    \\& =n+n_2+2(n-n_2-n_1)
    \\&\geq n+n_2+2(n-n_2-(2p-3)) 
    \\&= 3n-n_2-4p+6 
\end{align*}

    On the other hand, since $F$ is a forest:

    $$3n-n_2-4p+6\leq 2e(F)\leq 2(n-1),$$
    
    which implies $n_2 \geq n-4p$.

\end{proof}

\subsubsection{The coloring $\chi$ is switchable}\label{subsec: switch}

In this subsection we discuss the case when $\chi$ is switchable, meaning there are {~$p-1$} pairwise vertex-wise disjoint switcher $C_4$s which we will call $D_1,\cdots,D_{p-1}$.
\medskip

Since $n-4p\geq 3(p-1)$, it is not hard to see we may select $p-1$ vertices of $F$ with degree 2 $t_1,t_2,\cdots,t_{p-1}$, such that no two of them are or share neighbors.

Now, for each $i=1,2,\cdots,p-1$, let $a_i$ and $b_i$ be $t_i$'s neighbors. The main idea will be two embed each triple $(a_i,b_i,t_i)$ in three of $D_i$'s vertices.

\medskip

We may label each $D_i$'s vertices in consecutive order $d_{i,1},d_{i,2},d_{i,3}$ and $d_{i,4}$ in such way that $\chi(d_{i,4}d_{i,1})+\chi(d_{i,1}d_{i,2}) \neq \chi(d_{i,2}d_{i,3})+\chi(d_{i,3}d_{i,4})$.
\medskip

Now, call an embedding of $F$ in $K$ \textbf{good} if
\begin{itemize}
    \item $a_i=d_{i,2}$ for $i=1,2\cdots,p-1$;
    \item $b_i=d_{i,4}$ for $i=1,2,\cdots, p-1$;
    \item $t_i\in \{d_{i,1},d_{i,3}\}$ for $i=1,2,\cdots,p-1$.
\end{itemize}
\begin{rem}
    Notice a good embedding exists since $|K|\geq n+(p-1)$, with the extra $+(p-1)$ being added since each $t_i$ varies over $2$ possible vertices.
\end{rem}
We finish this case with the following case:

\begin{prop}\label{CauDevProf2}
    One of the good embeddings is zero-sum.
\end{prop}
\begin{proof}
    This will follow in a similar fashion to Proposition \ref{CauDavProf1}. 

    First of all, suppose all the other edges of $F$ that aren't of the form $a_it_i$ or $b_it_i$ are fixed. Let $s$ be their sum. 

    Let $B_i$ be the set of the possible values of the sum of the edges joining $t_i$ to its neighbors in all such good embeddings. Namely:
    $$|B_i| = \{\chi(d_{i,4}d_{i,1})+\chi(d_{i,1}d_{i,2}) ,\chi(d_{i,2}d_{i,3})+\chi(d_{i,3}d_{i,4})\}$$

    Notice that the possible values of the edge-sum of $F$ in such good embeddings $B$ is precisely the set $\{s\}+B_1+B_2+\cdots +B_{p-1}$. Therefore, using Theorem \ref{thm:CauDav} we derive the bound:
\begin{align*}
    |B|= |\{s\}+B_1+B_2+\cdots +B_{p-1}|&\geq \min (p,|\{s\}|+|B_1|+\cdots+|B_{p-1}|-(p-1)) \\&= \min(p,1+2+\cdots+2-(p-1)) = \min(p,p) \\&= p
\end{align*}
    Therefore $B$ must cover all $p$ residues modulo $p$, implying that some of the embeddings must be zero-sum.
\end{proof}
\begin{rem}\label{refnbushyswitch}
    In the case when $F$ is non-bushy but $\chi$ is switchable, we actually proved:

    $$R(F,\mathbb{Z}_p)\leq n+(p-1)$$

    whenever $n\geq 7p-3$.
\end{rem}
\subsubsection{The coloring $\chi$ is non-switchable}\label{subsec: non switch}

In this final case, we solve the case when $\chi$ is non-switchable, meaning the largest collection of pairwise vertex-disjoint switcher $C_4$ is at most $p-2$ in size. Excluding these, we can find a subgraph $K''$ of $K$ with $|K|-4(p-2)$ vertices that does not contain any switcher $C_4$.

Therefore, for all $(a,b,c,d) \in K''^4$ we must have that:
\begin{equation}\label{nonswitcheq}
    \chi(ab)+\chi(ad) = \chi (cb)+\chi(cd).
\end{equation}

We now use this property repeatedly to prove the following proposition:
\begin{prop}\label{moncrochK''}
    $K''$ is monochromatic.
\end{prop}

\begin{proof}
    Applying \ref{nonswitcheq} to $(a,b,c,d)$ and $(b,c,d,a)$ and subtracting both equations:
    
$$\begin{cases}
    \chi(ab)+\chi(ad) = \chi (cb)+\chi(cd).\\\\
    \chi(bc)+\chi(ba) = \chi (dc)+\chi(da).
\end{cases}    $$

    $$\implies \chi(ad) = \chi(bc).$$

Therefore, every two disjoint edges of $K''$ have the same color. 

Now, consider two edges $xy$ and $xz$ that share a vertex $x$ and consider a edge $tw$ disjoint to them. Because of what we just proved, $xy$ and $zw$ have the same color, as do $zw$ and $xz$, implying $xy$ and $xz$ have the same coloring and proving that neighboring edges also have the same color.

Thus, every pair of edges have the same color and therefore $K''$ is monochromatic.
\end{proof}

Hence, since $|K''| = |K| -(4p-2) \geq n$, $K''$ must contain a copy of $F$ which, by Proposition \ref{moncrochK''} is monochromatic and, therefore, zero-sum.
\begin{rem}\label{refnbushynswitch}
    In the case when $F$ is non bushy and $\chi$ is non-switchable, we actually proved:
    
    $$R(F,\mathbb{Z}_p)\leq n+4p-2$$

    without any size constrain on $n$ whatsoever.
\end{rem}
\section{Conclusion and open questions}\label{sec : conclu}

In this paper, we proved that $R(F, \mathbb{Z}_p) \leq n + c_p$ for  $c_p=9p-12$ and every forest $F$ on $n \geq n_p=3p^2-12p+11$ vertices without isolated vertices. In this section, we discuss some questions that naturally from our result. 

\medskip

We believe that for large forests the optimal constant would be $c_p = p-1$, meaning that:
\begin{conj}
For any prime $p$, there is a positive integer $n_p = n(p)$ such that for all forests $F$ on $n\geq n_p$ vertices with $p\mid e(F)$, we have
    $$R(F,\mathbb{Z}_p) \leq n+p-1.$$
\end{conj}

In fact, we can prove $c_p\geq p-1$ with some simple constructions:
\begin{itemize}
    \item for $p\geq 3$, take $F$ as a $(n-1)-$pointed star and construct the clique $K$ with $|K|  = n+p-2$ by first taking a $(p-1)$-regular graph with $n+p-2$ vertices (its existence is proven in  with all edges colored $1$ and adding all the remaining edges to form the clique, all with color $0$. Since the sum of any set of $n-1$ edges emanating from any vertex in $K$ is, by pigeonhole principle, at least $1$ and at most $p-1$, there can never be a zero-sum $(n-1)-$pointed star;
    \item for $p=2$, see \cite{caro1994complete}.
\end{itemize}

Furthermore, as pointed in remarks \ref{refbushyvib} and \ref{refnbushyswitch}, one can see that in both the cases when $F$ is vibrant \ref{subsec: vib} and when $\chi$ is switchable \ref{subsec: switch} the arguments used in this paper yield this optimal bound. Interestingly enough, both their counterparts, cases \ref{subsec: nvibrant} and \ref{subsec: non switch}, only happen when $K$ has a large (of size $n+\mathcal{O}(p)$) subset $K'$ such that every $v\in K'$ has a dominant color (in $\ref{subsec: non switch}$ $K'$ is monochromatic).
\medskip

Another natural conjecture that follows from the results in this paper is the generalization of the result for any graph $G$, meaning:
\begin{conj}
    For any prime $p$ there exists an integer $c_p = c(p)$ such that, for all $n$ and graphs $G$ on $n$ vertices with $p\mid E(G)$
    $$R(G,\mathbb{Z}_p) \leq n+c_p$$
\end{conj}

Finally, we still believe that the statement holds true when working in composite moduli, meaning:
\begin{conj}
    For any positive integer $k$ there exists an integer $c_k = c(k)$ such that, for all $n$ and graphs $G$ on $n$ vertices with $k\mid E(G)$
    $$R(G,\mathbb{Z}_k) \leq n+c_k.$$
\end{conj}

\section*{Acknowledgements}

L.~Colucci was supported by FAPESB (EDITAL FAPESB Nº 012/2022 - UNIVERSAL -
NºAPP0044/2023). FAPESB is the Bahia Research Foundation.

\bibliographystyle{acm}
\bibliography{ref}

\end{document}